\documentclass[12pt, reqno]{amsart}
\usepackage{amsmath, amsthm, amscd, amsfonts, amssymb, graphicx, color}
\usepackage[bookmarksnumbered, colorlinks, plainpages]{hyperref}
\hypersetup{colorlinks=true,linkcolor=red, anchorcolor=green, citecolor=cyan, urlcolor=red, filecolor=magenta, 
pdftoolbar=true}

\newtheorem{theorem}{Theorem}[section]
\newtheorem{lemma}[theorem]{Lemma}
\newtheorem{proposition}[theorem]{Proposition}
\newtheorem{corollary}[theorem]{Corollary}
\theoremstyle{definition}
\newtheorem{definition}[theorem]{Definition}
\newtheorem{example}[theorem]{Example}

\theoremstyle{remark}

\numberwithin{equation}{section}

\begin{document}

\setcounter{page}{1}

\title[Short Title]{approximate semi-amenability of Banach algebras}

\author[M. Shams Kojanaghi]{M. Shams Kojanaghi}
\address[M. Shams Kojanaghi]{Department of Mathematics Science and Research Branch, Islamic Azad University
 Tehran of Iran}
\email{{mstafa.shams99@yahoo.com}}

\author[K. Haghnejad Azar]{K. Haghnejad Azar $^*$}
\address[K. Haghnejad Azar]{Department of Mathematics, University of Mohaghegh Ardabili, Ardabil of Iran.}
\email{haghnejad@uma.ac.ir}

\author[M. R. Mardanbeigi]{M. R. Mardanbeigi}
\address[M. R. Mardanbeigi]{Department of Mathematics Science and Research Branch, Islamic Azad University
 Tehran of Iran}
\email{mmardanbeigi@yahoo.com}

\let\thefootnote\relax\footnote{Copyright 2016 by the Tusi Mathematical Research Group.}



\date{Received: xxxxxx; Revised: yyyyyy; Accepted: zzzzzz.
\newline \indent $^{*}$Corresponding author}
\begin{abstract}
 Let $\mathfrak{A}$ be a Banach algebra, and $\mathcal{X}$ a Banach $\mathfrak{A}$-bimodule. A bounded
  linear mapping $\mathcal{D}:\mathfrak{A}\rightarrow \mathcal{X}$ is approximately semi-inner derivation if
   there eixist nets $(\xi_{\alpha})_{\alpha}$ and $(\mu_{\alpha})_{\alpha}$ in $\mathcal{X}$ such that, for each
    $a\in\mathfrak{A}$, $\mathcal{D}(a)=\lim_{\alpha}(a.\xi_{\alpha}-\mu_{\alpha}.a)$. $\mathfrak{A}$ is called
     approximately semi-amenable if for every Banach $\mathfrak{A}$-bimodule $\mathcal{X}$, every
      $\mathcal{D}\in\mathcal{Z}^{1}(\mathfrak{A},\mathcal{X}^{*})$ is approximtely semi-inner. There are
       some Banach algebras which are approximately semi-amenable, but not approximately amenable.
 In this manuscript, we investigate some properties of approximate semi-amenability of Banach algebras. Also in
  Theorem \ref{ee} we prove the approximate semi-amenability of Segal algebras on a locally compact group $G$.  
\end{abstract}
 \maketitle

\textbf{\section{\bf \textbf{Introduction and preliminaries}}}

Let $\mathfrak{A}$ be a Banach algebra, and $\mathcal{X}$ a Banach $\mathfrak{A}$-bimodule. A derivation is a
 linear map $\mathcal{D}:\mathfrak{A}\rightarrow \mathcal{X}$, such that $\mathcal{D}(ab)=\mathcal{D
 }(a)b+a\mathcal{D}(b)$, $(a\in\mathfrak{A})$.
Throughout, unless otherwise stated, by a derivation we mean a continuous derivation. 
 The space of continuous derivations from $\mathfrak{A}$ to $\mathcal{X}$ is denoted by $\mathcal{Z}^{1
 }(\mathfrak{A},\mathcal{X})$.
 For $x\in\mathcal{X}$, set $\delta_{x}(a):=a.x-x.a$, which is a continuous derivation from $\mathfrak{A}$ to
  $\mathcal{X}$, and is called inner derivation.
 The space of all inner derivations from $\mathfrak{A}$ to $\mathcal{X}$ is denoted by $\mathcal{N}^{1
 }(\mathfrak{A},\mathcal{X})$, the first cohomology group of $\mathfrak{A}$ with coefficients
in $\mathcal{X}$ is defined to be the quotient space $\mathcal{H}^{1}(\mathfrak{A},\mathcal{X}):=
 \mathcal{Z}^{1}(\mathfrak{A},\mathcal{X})/\mathcal{N}^{1}(\mathfrak{A},\mathcal{X})$.

The Banach algebra $\mathfrak{A}$ is amenable if and only if $\mathcal{H}^{1}(\mathfrak{A},
 \mathcal{X}^{*}):= \left\lbrace 0\right\rbrace$ for each Banach $\mathfrak{A}$-bimodule $\mathcal{X}$, and
  that $\mathfrak{A}$ is contractible if and only if $\mathcal{H}^{1}(\mathfrak{A},\mathcal{X}):=\left\lbrace
   0\right\rbrace $ for each Banach $\mathfrak{A}$-bimodule $\mathcal{X}$.

Gourdeau in \cite{5}, showed that a Banach algebra $\mathfrak{A}$ is amenable if and only if every bounded
 derivation $\mathcal{D}:\mathfrak{A}\rightarrow \mathcal{X}$, for any Banach $\mathfrak{A}$-bimodule
  $\mathcal{X}$ can be approximated by a net of inner derivations [\cite{5}, Proposition 1]. A weaker version of
   this notion is approximate amenability of Banach algebras, that is, a Banach algebras $\mathfrak{A}$ is
    approximately amenable if and only if every bounded derivation $\mathcal{D}:\mathfrak{A}\rightarrow
     \mathcal{X}^{*}$ for any Banach dual $\mathfrak{A}$-bimodule $\mathcal{X}^{*}$ can be approximated by a
      net of inner derivations.
The concept of \emph{approximate amenability} of Banach algebras first was introduced and studied by
Ghahramani and Loy in \cite{1}. Then they showed that approximate amenability and approximate contractibility of
 Banach algebras are equivalent [\cite{4}, Theorem 2.1]. They also introduced the concept of 
 \textit{semi-amenability} and \emph{approximate semi-amenability} of Banach algebras in \cite{2}, and studied
  the last notion on tensor product of Banach algebras.\\ 
Let $\mathfrak{A}$ be a Banach algebra and $\mathcal{X}$ be a Banach $\mathfrak{A}$-bimodule. A derivation
$\mathcal{D}:\mathfrak{A}\rightarrow \mathcal{X}$ is called \textit{approximately inner} if there exists a net
 $(\xi_{\alpha})$ in $\mathcal{X}$ such that $\mathcal{D}(a)=\lim_{\alpha}(a.\xi_{\alpha}-\xi_{\alpha}.a)$, for
  each $a\in\mathfrak{A}$.
 $\mathcal{D}$ is \textit{approximately semi-inner} if there are nets $(\zeta_{\alpha})_{\alpha}$ and
  $(\mu_{\alpha})_{\alpha}$ in $\mathcal{X}$ such that $\mathcal{D}(a)=\lim_{\alpha}(a
  .\zeta_{\alpha}-\mu_{\alpha}.a)$, for each $a\in\mathfrak{A}$.
 
A Banach algebra $\mathfrak{A}$ is \textit{approximately amenable} (rep: \textit{approximately semi-amenable}) if  
 for every Banach $\mathfrak{A}$-bimodule $\mathcal{X}$, every derivation $\mathcal{D}
 :\mathfrak{A}\rightarrow \mathcal{X}^{*}$ is approximately inner (rsp: approximately semi-inner).\\
For a Banach algebra $\mathfrak{A}$ the projective tensor product $\mathfrak{A}\widehat{\otimes} \mathfrak{A}$
 is the completion of algebraic tensor product $\mathfrak{A}\otimes \mathfrak{A}$ with respect to projective tensor
  norm, and is a Banach $\mathfrak{A}$-bimodule by the multiplication specified by
\begin{align*}
&a.(b\otimes c):=ab\otimes c,\\
&(b\otimes c).a:=b\otimes ca, \;\;\;\; (a ,b , c\in \mathfrak{A}).
\end{align*}
The dual space $(\mathfrak{A}\widehat{\otimes}\mathfrak{A})^{*}$ is $BL(\mathfrak{A}\times\mathfrak{A})$,
 the space of bounded bilinear forms on $\mathfrak{A}\times \mathfrak{A}$. The product map and it's first and
  second duals, respectively, specify by
\begin{align*}
&\pi:\mathfrak{A}\widehat{\otimes}\mathfrak{A}\rightarrow \mathfrak{A},\;\;\pi(a\otimes b)=ab,\\
&\pi^{*}:\mathfrak{A}^{*}\rightarrow (\mathfrak{A}\widehat{\otimes}\mathfrak{A})^{*},\;\;\langle\pi^{*
}(f),a\otimes b\rangle=\langle f,ab\rangle,\\
&\pi^{**}:(\mathfrak{A}\widehat{\otimes}\mathfrak{A})^{**}\rightarrow\mathfrak{A}^{**},\;\;\langle \pi^{**
}(F''),f\rangle=\langle F'',\pi^{*}(f)\rangle,
\end{align*}
for $ a, b \in \mathfrak{A}$, $f\in \mathfrak{A}^{*}$ and
 $F''\in(\mathfrak{A}\widehat{\otimes}\mathfrak{A})^{**}$.
Clearly, $\pi$ is an $\mathfrak{A}$-bimodule homomorphism with respect to the above module structure on
 $\mathfrak{A}\widehat{\otimes}\mathfrak{A}$.\\
Also for Banach algebras $\mathfrak{A}$ and $\mathcal{B}$, the $l_{1}$-direct sum $\mathfrak{A}\oplus
 \mathcal{B}$ equipped with the multiplication $(a,b).(c,d):=(ac,bd)$ is a Banach algebra.\\
 A Banach algebra $\mathfrak{A}$ is said to be \textit{dual} if there is a closed submodule $\mathfrak{A}_{*}$ of
 $\mathfrak{A}^{*}$, such that $\mathfrak{A}=(\mathfrak{A}_{*})^{*}$. For a dual Banach algebra
  $\mathfrak{A}$, a dual Banach $\mathfrak{A}$-bimodule 
 $\mathcal{X}$ is called \textit{normal} if, for each $x\in \mathcal{X}$ the maps $a\rightarrow a.x$ and
   $a\rightarrow x.a$  from $\mathfrak{A}$ into $\mathcal{X}$, are $w^{*}$-continuous.
 A dual Banach algebra $\mathfrak{A}$ is \textit{Connes-amenable} if, for every normal, dual
 Banach $\mathfrak{A}$-bimodule $\mathcal{X}$, every $w^{*}$-continuous derivation $\mathcal{D}:
\mathfrak{A}\rightarrow\mathcal{X}$ is inner. We denote $\mathcal{Z}^{1}_{w^{*}}(\mathfrak{A}
  ,\mathcal{X})$ for the $w^{*}$-continuous derivations from $\mathfrak{A}$ into $\mathcal{X}$ and
   $\mathcal{H}^{1}_{w^{*}}(\mathfrak{A},\mathcal{X})=\mathcal{Z}^{1}_{w^{*}}(\mathfrak{A}
  ,\mathcal{X})/\mathcal{N}^{1}(\mathfrak{A},\mathcal{X})$.\\
The unitization of non-unital Banach algebra $\mathfrak{A}$, is denoted by
 $\mathfrak{A}^{\sharp}=\mathfrak{A}\oplus\mathbb{C}$. 
We will sometimes abbreviate the phrases ''bounded approximate identity'' and ''approximate
  identity'' to B.A.I and A.I, respectively.
  
\textbf{\section{\bf Main results}}
A Banach algebra $\mathfrak{A}$ is approximately amenable (rep: approximately semi-amenable) if for every
 Banach $\mathfrak{A}$-bimodule $\mathcal{X}$, every derivation $\mathcal{D}:\mathfrak{A}\rightarrow
  \mathcal{X}^{*}$ is approximately inner (rsp: approximately semi-inner).
For $(1\leq p<\infty)$, the  Banach sequence algebras $l^{p}=l^{p}(\mathbb{N})$, are neither amenable nor
 approximately amenable [\cite{7}, Theorem 4.1]. But Example 3.7 of \cite{2} shows that $l^{p}$ is
  approximately semi-amenable whenever $(1\leq p<\infty)$.\\
Consequently, a Banach algebra can be approximately semi-amenable, without being approximately amenable, and this motivates us to
 decide that all chosen Banach algebras throughout this paper not to be approximately amenable, unless otherwise specified. 
    \begin{proposition} $\label{s.m}$
    Let $\mathfrak{A}$ be a Banach algebra. Then $\mathfrak{A}$ is approximately semi-amenable if
     and only if $\mathfrak{A}^{\sharp}$ is approximately semi-amenable.
  \end{proposition}
  \begin{proof} 
 Let $\mathfrak{A}$ be approximately semi-amenable. Suppose that $\mathcal{X}$ is a Banach
  $\mathfrak{A}^{\sharp}$-bimodule and $\mathcal{D}:\mathfrak{A}^{\sharp}\rightarrow \mathcal{X}^{*}$ a
   continuous derivation. Then by [\cite{1}, Lemma 2.3], $\mathcal{D}= \mathcal{D}_{1}+\delta_{\lambda}$,
    where $\lambda\in \mathcal{X}^{*}$ and 
  $\mathcal{D}_{1}:\mathfrak{A}^{\sharp}\rightarrow e.\mathcal{X}^{*}.e$ is a derivation. Since
   $\mathcal{D}_{1}(e)=0$ and $\mathcal{D}_{1}\mid_{\mathfrak{A}}$ is approximately semi-inner, so is
    $\mathcal{D}$.\\
 Conversely, let   $\mathfrak{A}^{\sharp}$ be approximately semi-amenable. Suppose that $\mathcal{X}$ is a Banach
  $\mathfrak{A}$-bimodule and $\mathcal{D}:\mathfrak{A}\rightarrow \mathcal{X}^{*}$ a continuous derivation.
So $\mathcal{X}$ is a Banach $\mathfrak{A}^{\sharp}$-bimodule by module actions,
 
$ (a+\lambda e).x:=ax,\;\;x.(a+\lambda e):=xa, \;\; ((a+\lambda e)\in\mathfrak{A}^{\sharp})$.

  Now define 
  \begin{align*}
 & \mathcal{D}^{\thicksim}:\mathfrak{A}^{\sharp}\rightarrow\mathcal{X}^{*}\\
  & \mathcal{D}^{\thicksim}(a+\lambda e)=\mathcal{D}(a),\;\;(a\in\mathfrak{A}, \lambda\in\mathbb{C}).
  \end{align*}
 Obviously, $\mathcal{D}^{\thicksim}$ is a continuous derivation. Thus by assumption there are nets
  $(\xi_{\alpha})$ and $(\mu_{\alpha})$ in $\mathcal{X}^{*}$, such that
  \begin{align*}
  \mathcal{D}(a)=\mathcal{D}^{\thicksim}(a+\lambda e)&=\lim_{\alpha}((a+\lambda e).
  \xi_{\alpha}-\mu_{\alpha}.(a+\lambda e))\\
  &=\lim_{\alpha}(a\xi_{\alpha}-\mu_{\alpha}a).
  \end{align*}
 Thus $\mathfrak{A}$ is approximately semi-amenable.
  \end{proof}

 \begin{lemma}
Let the Banach algebra $\mathfrak{A}$ be approximately semi-amenable. Then $\mathfrak{A}$ has an approximate identity.
\end{lemma}
\begin{proof}
The proof is similar to Lemma 2.2 in \cite{1}.
\end{proof}
 \begin{theorem}\label{mm}
 The Banach algebra $\mathfrak{A}$ is approximately semi-amenable if and only if there exist $($ necessarily 
 nonequal $)$ nets $(M_{\alpha})$ and $(N_{\alpha})$ in
  $(\mathfrak{A}^{\sharp}\widehat{\otimes}\mathfrak{A}^{\sharp})^{**}$ such that for
   each  $a\in\mathfrak{A}^{\sharp}$
  \begin{align*}
 (i)\;\;\;& a.M_{\alpha}-N_{\alpha}.a\longrightarrow 0, \\
 (ii)\;\; &\pi^{**}_{_{\mathfrak{A}^{\sharp}}}(M_{\alpha})=\pi_{_{\mathfrak{A}^{\sharp}}}^{**
 }(N_{\alpha})=e_{_{(\mathfrak{A}^{\sharp})^{**}}},\;\;\; (\alpha\in\Lambda).
 \end{align*}
 \end{theorem}
 \begin{proof}
 If $\mathfrak{A}$ is approximately semi-amenable, then by Proposition $\ref{s.m}$ so is
  $\mathfrak{A}^{\sharp}$. Suppose that $\tau:
  \mathfrak{A}^{\sharp}\widehat{\otimes}\mathfrak{A}^{\sharp}\longrightarrow 
  (\mathfrak{A}^{\sharp}\widehat{\otimes}\mathfrak{A}^{\sharp})^{**}$ and $\tau_{1}
  :\mathfrak{A}^{\sharp}\longrightarrow(\mathfrak{A}^{\sharp})^{**}$ be natural injections.
   Set $u=\tau(e\otimes e)$. Define
  \begin{align*}
 & \delta_{u}:
 \mathfrak{A}^{\sharp}\longrightarrow(\mathfrak{A}^{\sharp}\widehat{\otimes}\mathfrak{A}^{\sharp})^{**}\\
  &\delta_{u}(a)=a.u-u.a,\;\;   (a\in\mathfrak{A}^{\sharp}).
  \end{align*}
So $\delta_{u}$ is a continuous derivation, and for each $f\in(\mathfrak{A}^{\sharp})^{*}$,
\begin{align*}
\langle\pi^{**}(\delta_{u}(a)),f\rangle &=\langle \pi^{**}(a.\tau(e\otimes e)-\tau(e\otimes e).a),f\rangle\\
&=\langle a.\tau(e\otimes e)-\tau(e\otimes e).a,\pi^{*}(f)\rangle\\
&=\langle a.\tau(e\otimes e),\pi^{*}(f)\rangle-\langle \tau(e\otimes e).a,\pi^{*}(f)\rangle\\
&=\langle\tau(e\otimes e),\pi^{*}(f).a\rangle -\langle \tau(e\otimes e), a.\pi^{*}(f)\rangle\\
&=\langle\pi^{*}(f).a,e\otimes e\rangle -\langle a.\pi^{*}(f) ,e\otimes e\rangle\\
&=\langle\pi^{*}(f),ae\otimes e\rangle -\langle \pi^{*}(f) ,e\otimes ea\rangle\\
&=\langle f,\pi(a.e\otimes e-e\otimes e.a)\rangle =\langle f,0\rangle.
\end{align*}  
Thus for each $a\in\mathfrak{A}^{\sharp}$, $\delta_{u}(a)\in ker\pi^{**}$. Since $ker\pi^{**}$ is a Banach 
  $\mathfrak{A}^{\sharp}$-bimodule, therefore $\delta_{u}\in\mathcal{Z}^{1
  }(\mathfrak{A}^{\sharp},ker\pi^{**})$, and by assumption there are nets $(m_{\alpha})$ and $(n_{\alpha})$ in
   $ker\pi^{**}$ such that for each $a\in\mathfrak{A}^{\sharp}$, $\delta_{u}(a)=\lim_{\alpha}(a.m_{\alpha}
   -n_{\alpha}.a)$. Now we set $(M_{\alpha})_{\alpha}=(u-m_{\alpha})_{\alpha}$ and
      $(N_{\alpha})_{\alpha}=(u-n_{\alpha})_{\alpha}$, which are nets in 
      $(\mathfrak{A}^{\sharp}\widehat{\otimes}\mathfrak{A}^{\sharp})^{**}$.  It is easily verified that
       $\pi^{**}\circ \tau=\tau_{1}\circ \pi$, then for each $a\in \mathfrak{A}^{\sharp}$ we have
 \begin{enumerate}
\item $a.M_{\alpha}-N_{\alpha}.a=a.(u-m_{\alpha})-(u-n_{\alpha}).a=\delta_{u}(a)-(a.m_{\alpha}-n_{\alpha}.a)
 \longrightarrow 0$,
 \item $\pi^{**}_{\mathfrak{A}^{\sharp}}(M_{\alpha})=\pi^{**}_{\mathfrak{A}^{\sharp}}(u
 -m_{\alpha})=\pi^{**}_{\mathfrak{A}^{\sharp}}(u)=\pi^{**}_{\mathfrak{A}^{\sharp}}(\tau(e\otimes
  e))=\tau_{1}(\pi(e\otimes e))=e_{_{(\mathfrak{A}^{\sharp})^{**}}}$,
 \item $\pi^{**}_{\mathfrak{A}^{\sharp}}(N_{\alpha})=\pi^{**}_{\mathfrak{A}^{\sharp}}(u
 -n_{\alpha})=\pi^{**}_{\mathfrak{A}^{\sharp}}(u)=\pi^{**}_{\mathfrak{A}^{\sharp}}(\tau(e\otimes
  e))=\tau_{1}(\pi(e\otimes e))=e_{_{(\mathfrak{A}^{\sharp})^{**}}}$.
\end{enumerate} 
Conversely, suppose that ''only if'' holds. For each $\alpha$, according to Goldstine's Theorem, there are bounded
 nets $(m^{\lambda}_{\alpha})_{\lambda}$ and $(n^{\lambda}_{\alpha})_{\lambda}$ in
  $\mathfrak{A}^{\natural}\widehat{\otimes}\mathfrak{A}^{\natural}$, such that, 
\begin{align*}
m^{\lambda}_{\alpha}\longrightarrow M_{\alpha}, \;\;n^{\lambda}_{\alpha}\longrightarrow N_{\alpha}
\end{align*}
in weak$^{*}$-topology on $(\mathfrak{A}^{\natural}\widehat{\otimes}
\mathfrak{A}^{\natural})^{**}$. Let $\mathcal{X}$ be a neo-unital Banach
 $\mathfrak{A}^{\natural}$-bimodule and $\mathcal{D}\in\mathcal{Z}^{1}(\mathfrak{A}^{\natural} 
 ,\mathcal{X}^{*})$ be continuous. We show that $\mathcal{D}$ is approximately semi-inner. Let 
 $ m^{\lambda}_{\alpha}=
  \sum_{n=1}^{\infty} a_{n,\alpha}^{(\lambda)}\otimes b_{n,\alpha}^{(\lambda)}$ and $
   n^{\lambda}_{\alpha}= \sum_{n=1}^{\infty} c_{n ,\alpha}^{(\lambda)}\otimes d_{n,\alpha}^{(\lambda)}$
    with $\sum_{n=1}^{\infty}\parallel  a_{n,\alpha}^{(\lambda)}\parallel b_{n,
    \alpha}^{(\lambda)}\parallel<\infty$ and $\sum_{n=1}^{\infty}\parallel c_{n
    ,\alpha}^{(\lambda)}\parallel\parallel d_{n,\alpha}^{(\lambda)}\parallel<\infty$. Then $(\sum_{n=1}^{\infty} 
     a_{n,\alpha}^{(\lambda)}.\mathcal{D}b_{n,\alpha}^{(\lambda)})_{\lambda}$ and $(\sum_{n=1}^{\infty}
      c_{n,\alpha}^{(\lambda)}.\mathcal{D}d_{n,\alpha}^{(\lambda)})_{\lambda}$ are bounded nets in
       $\mathcal{X}^{*}$, which have $w^{*}$-accumulation points such as $\eta_{\alpha}$ and $\xi_{\alpha}$ in
        $\mathcal{X}^{*}$, respectively. Without loss of generality, we suppose that $\eta_{\alpha}$ and
         $\xi_{\alpha}$ are respectively, the $w^{*}$-limits of $(\sum_{n=1}^{\infty}  a_{n,\alpha}^{(\lambda)}
         .\mathcal{D}b_{n
   ,\alpha}^{(\lambda)})_{\lambda}$ and $(\sum_{n=1}^{\infty} c_{n,\alpha}^{(\lambda)}.\mathcal{D}d_{n
   ,\alpha}^{(\lambda)})_{\lambda}$. Define
 \begin{equation}
  \psi:\mathfrak{A}^{\sharp}\widehat{\otimes}\mathfrak{A}^{\sharp} \longrightarrow \mathcal{X}^{*}, \;\;
   \psi(a\otimes b)=a\mathcal{D}(b),\label{3.4}
 \end{equation}
 which is a bilinear and continuous map. Also by assumption we have 
 \begin{equation}
  a.M_{\alpha}-N_{\alpha}.a\longrightarrow 0\label{3.5}
 \end{equation}
 in weak$^{*}$-topology on $(\mathfrak{A}^{\natural}\widehat{\otimes}\mathfrak{A}^{\natural})^{**}$.
Then, from \eqref{3.4}, (\ref{3.5}), and neo-unitality of $\mathcal{X}$, for each
 $a\in\mathfrak{A}^{\sharp}$ and $x\in\mathcal{X}$, we have
\begin{align*}
\langle x,a.\eta_{\alpha} \rangle &=\lim_{\lambda}\langle x
,\sum_{n=1}^{\infty}aa_{n,\alpha}^{(\lambda)}\mathcal{D} b_{n,\alpha}^{(\lambda)}\rangle\\
&=\lim_{\lambda}\langle x,\psi(a.m_{\alpha}^{\lambda})\rangle\\
&=\lim_{\lambda}\langle x,\psi(n_{\alpha}^{\lambda}.a)\rangle\\
&=\lim_{\lambda}\langle x,\sum_{n=1}^{\infty}  c_{n,\alpha}^{(\lambda)}\mathcal{D}(d_{n,
\alpha}^{(\lambda)}a)\rangle\\
&=\lim_{\lambda}\langle x,\sum_{n=1}^{\infty} (c_{n,\alpha}^{(\lambda)}d_{n,
\alpha}^{(\lambda)}\mathcal{D}a+c_{n,\alpha}^{(\lambda)}\mathcal{D}(d_{n,\alpha}^{(\lambda)})a)\rangle\\
&=\lim_{\lambda}\langle x.\sum_{n=1}^{\infty} c_{n,\alpha}^{(\lambda)}d_{n,\alpha}^{(\lambda)}
,\mathcal{D}a\rangle+\lim_{\lambda}\langle x,\sum_{n=1}^{\infty} c_{n,\alpha}^{(\lambda)}\mathcal{D}(d_{n,
\alpha}^{(\lambda)})a\rangle\\
&=\lim_{\lambda}\langle x.\sum_{n=1}^{\infty} c_{n,\alpha}^{(\lambda)}d_{n,\alpha}^{(\lambda)} 
,\mathcal{D}a\rangle+\langle x,\xi_{\alpha} .a\rangle.
\end{align*}
Consequently, we have
$$\lim_{\alpha}\langle x,a.\eta_{\alpha} \rangle=\lim_{\alpha}\lim_{\lambda}\langle x.\sum_{n=1}^{\infty}
 c_{n,\alpha}^{(\lambda)}d_{n,\alpha}^{(\lambda)},\mathcal{D}a\rangle+\lim_{\alpha}\langle x
 ,\xi_{\alpha}.a\rangle$$
and therefore $\lim_{\alpha}\langle x,a.\eta_{\alpha} \rangle=\langle x,\mathcal{D}(a)
 \rangle+\lim_{\alpha}\langle x,\xi_{\alpha}.a \rangle$.
It follows that $\mathcal{D}$ is approximately semi-inner and consequently, by Proposition \ref{s.m}
 $\mathfrak{A}$ is approximately semi-amenable.  
 \end{proof}

 Similarly, we have the following parallel result.
 \begin{theorem}\label{aa}
 The Banach algebra $\mathfrak{A}$ is approximately semi-contractible if and only if there exist $($ necessarily 
 nonequal $)$ nets $(M_{\alpha})$ and $(N_{\alpha})$ in
  $(\mathfrak{A}^{\sharp}\widehat{\otimes}\mathfrak{A}^{\sharp})$ such
   that for each  $a\in\mathfrak{A}^{\sharp}$
  \begin{align*}
 (i)\;\;\;& a.M_{\alpha}-N_{\alpha}.a\longrightarrow 0, \\
 (ii)\;\; &\pi(M_{\alpha})\longrightarrow e,\\
 (iii)\;\;&\pi(N_{\alpha})\longrightarrow e.
 \end{align*}
 \end{theorem}
 \begin{proof}
 The proof is similar to Theorem \ref{mm}.
 \end{proof}

From Theorems \ref{mm} and \ref{aa}  and 
$\overline{\mathfrak{A}^{\sharp}\widehat{\otimes}\mathfrak{A}^{\sharp}}^{w^{*}}
=(\mathfrak{A}^{\sharp}\widehat{\otimes}\mathfrak{A}^{\sharp})^{**}$ as a
 consequence from Goldstines Theorem, we conclude that approximately semi-amenability and approximately 
 semi-contractibility of a Banach
  algebra are equivalent.

\begin{definition}
Let $G$ be a locally compact group. A subset $S^{1}(G)$ of the Banach algebra $L^{1}(G)$ is said to be a
\textit{ Segal algebra} if it satisfies the following conditions:
\begin{enumerate}
\item $S^{1}(G)$ is dense in $L^{1}(G)$.
\item $S^{1}(G)$ is a Banach space, with norm $\Vert.\Vert_{s}$, and $$\Vert f \Vert_{1}\leq \Vert f \Vert_{s},\;\;\
;(f\in S^{1}(G)).$$
\item $ S^{1}(G)\ni f\Longrightarrow L_{x}*f\in S^{1}(G)$ for all $x\in G$ where, $$L_{x}*f(y)=f(x^{-1}y),\;\;
 (y\in G),$$ and the mapping $x\rightarrow L_{x}*f$ of $G$ into $S^{1}(G)$ is continuous.
\item $f\in S^{1}(G)\Longrightarrow \Vert L_{x}*f \Vert _{s}=\Vert f \Vert_{s}$ for all $x\in G$.
\end{enumerate} 
It readily follows from the above conditions and Proposition 1 in Section 4 of \cite{8}, that $S^{1}(G)$ with
 convolution as multiplication, is a \textit{left ideal} of the algebra $L^{1}(G)$, and that 
\begin{equation}
\Vert h*f \Vert_{s}\leq \Vert h \Vert_{1}\Vert f \Vert_{s} \label{1.1}
\end{equation}
 holds for all $f\in S^{1}(G)$ and all $h\in L^{1}(G)$, in particular it follows that $S^{1}(G)$ is in fact, a Banach
 algebra.
 \end{definition}
 \begin{definition}
 A locally compact group $G$ is called a [SIN] group if there exists a topological basis of conjugate invariant
  neighbourhoods of the identity element of the group $G$.
  \end{definition}

\begin{theorem} \label{ee}
Every proper Segal algebra of a locally compact group $G$ is approximately semi-amenable. 
\end{theorem}
\begin{proof}
Let $G$ be a locally compact group. Suppose that $S^{1}(G)$ is a proper Segal subalgebra of $L^{1}(G)$, and $(f_{n})_{n}$ is a sequence in $L^{1}(G)$. Since $\overline{S^{1}(G)}^{^{\Vert.\Vert_{1}}}=L^{1}(G)$, so for each $n\in\mathbb{N}$ there exists a sequence $(f_{n,m})_{m}\subset S^{1}(G)$ such that $f_{n,m}\rightarrow f_{n}$ in norm topology of $L^{1}(G)$. Take $f_{0}\in S^{1}(G)$ fixed, then $f_{n,m}*f_{0}\longrightarrow f_{n}*f_{0}.$ Set $$F_{n, m}=(f_{n,m}-f_{n})*f_{0},$$
then $(F_{n,m})_{m}=\left\lbrace (f_{n,m}-f_{n})*f_{0}\right\rbrace_{m} \subset S^{1}(G)$,
 and by (\ref{1.1})$$\Vert F_{n,m}\Vert_{s}=\Vert(f_{n,m}-f_{n})*f_{0}\Vert_{s}\leq \Vert f_{n,m}
 -f_{n}\Vert_{1}\Vert f_{0}\Vert_{s}.$$ So $\lim_{m}  F_{n,m}=0$, in norm-topology of $S^{1}(G)$. Let 
 \begin{align*}
 M_{m}=\frac{6}{\pi^{2}}\sum_{n=1}^{\infty}\frac{(F_{n,m}+e)}{n^{2}} \otimes e ,\\
  N_{m}=\frac{6}{\pi^{2}}\sum_{n=1}^{\infty} e \otimes\frac{(F_{n,m}+e)}{n^{2}} .
 \end{align*}
 Obviously, $$(M_{m}), (N_{m}) \subset S^{1}(G)^{\sharp}\widehat{\otimes} S^{1}(G)^{\sharp}.$$
 For each $H=(f+\lambda e)\in S^{1}(G)^{\sharp}$, we have 
 \begin{align*}
 (i)\; H.M_{m}-N_{m}.H&=\frac{6}{\pi^{2}}\sum_{n=1}^{\infty} H.\frac{(F_{n,m}+e)}{n^{2}}\otimes e-e
  \otimes\frac{(F_{n,m}+e)}{n^{2}}.H\\
 &=\frac{6}{\pi^{2}}\sum_{n=1}^{\infty} (f+\lambda e).\frac{(F_{n,m}+e)}{n^{2}}\otimes e-e \otimes\frac{(F_{n,m}+e)
 }{n^{2}}.(f+\lambda e)\\ 
 &=\frac{6}{\pi^{2}}\sum_{n=1}^{\infty}\frac{(f.F_{n,m},\lambda)}{n^{2}}\otimes e-e\otimes\frac{(F_{n,m}.f,\lambda)
 }{n^{2}}\\
 &\rightarrow\frac{6}{\pi^{2}}\sum_{n=1}^{\infty} [\frac{\lambda .e\otimes e-e\otimes \lambda.e}{n^{2}}]=0
 \end{align*}
 \begin{align*}
 (ii)\;\; \pi_{_{S^{1}(G)^{\sharp}}}(M_{m})&=\frac{6}{\pi^{2}}\sum_{n=1}^{\infty} \frac{(F_{n,m}+e)}{n^{2}}.e\\
 &\rightarrow\frac{6}{\pi^{2}} \sum_{n=1}^{\infty}\frac{e}{n^{2}}=e.\\
 (iii)\;\; \pi_{_{S^{1}(G)^{\sharp}}}(N_{m})&=\frac{6}{\pi^{2}}\sum_{n=1}^{\infty} e. \frac{(F_{n,m}+e)}{n^{2}}\\
 &\rightarrow \frac{6}{\pi^{2}}\sum_{n=1}^{\infty}\frac{e}{n^{2}}=e.
 \end{align*}
 Hence by Theorems \ref{aa} and \ref{mm} $ S^{1}(G)$ is approximately semi-amenable. 
\end{proof}
\textit{Recall}: when $G$ is SIN group this kind of Segal algebras are not approximately amenable \cite{9}.
\begin{example} Let $\mathbb{R}$ be the real numbers group, which is a locally compact, Abelian and SIN group. Then every
 proper Segal subalgebra of $L^{1}(\mathbb{R})$ is approximately semi-amenable. For instance
\begin{enumerate}
\item $S^{1}(\mathbb{R})$ the Banach space of all continuous functions in $L^{1}(\mathbb{R})$ which vanish at
 infinity, with norm $\|f\|_{s}=\|f\|_{1}+\|f\|_{\infty}$ is a Segal algebra \cite{8}.
\item $S^{1}(\mathbb{R})=L^{1}(\mathbb{R})\cap L^{p}(\mathbb{R})$, $1<p<\infty$, with norm $\|f\|_{s}=
\|f\|_{1}+\|f\|_{p}$ is a Segal algebra on $\mathbb{R}$, \cite{8}. 
\end{enumerate}
\end{example}
Thus both of above Segal algebras are  approximately semi-amenable which by Theorem 1 of \cite{9} are never
 approximately amenable.\\

 Now consider the following Lemmas which we use them in the some parts of this section.
  \begin{lemma}$\label{ss}$
Let $\mathfrak{A}$ be a Banach algebra. Then there is a continuous linear mapping $\Psi
:\mathfrak{A}^{**}\widehat{\otimes}\mathfrak{A}^{**}\rightarrow
 (\mathfrak{A}\widehat{\otimes}\mathfrak{A})^{**}$ such that for $ a, b, x \in\mathfrak{A}$  and $ m \in
  \mathfrak{A}^{**}\widehat{\otimes}\mathfrak{A}^{**}$  the following hold:\\
$(i)$\;\; $\Psi(a \otimes b)=a\otimes b$,\;\;\;\;\;\;\;$(ii)$\;\; $\Psi(m).x=\Psi(m.x)$, \\
 $(iii)$\;\; $x.\Psi(m)=\Psi(x.m)$,\;\;$(iv)$\;\; $(\pi_{\mathfrak{A}})^{**} (\Psi(m))=
 \pi_{\mathfrak{A}^{**}}(m)$.
\end{lemma}
\begin{proof}
See \cite{3}.
\end{proof}

 \begin{lemma}\label{sm}
Let $\mathfrak{A}$ and $\mathcal{B}$ be Banach algebras. Then there is a continuous linear mapping
 $\Psi:\mathfrak{A}^{**}\widehat{\otimes}\mathcal{B}^{**}\rightarrow
  (\mathfrak{A}\widehat{\otimes}\mathcal{B})^{**}$ such that for $a\in\mathfrak{A}$, $b\in \mathcal{B}$, 
  $x\in \mathfrak{A}\widehat{\otimes}\mathcal{B}$  and $m\in 
  \mathfrak{A}^{**}\widehat{\otimes}\mathcal{B}^{**}$  the following hold:\\
$(i)$\;\; $\Psi(a\otimes b)=a\otimes b$,\\
$(ii)$\;\;$\Psi(m).x=\Psi(m.x)$,\\
$(iii)$\;\;$ x.\Psi(m)=\Psi(x.m)$.
\end{lemma}
\begin{proof}
The proof is similar to Lemma \ref{ss}.
\end{proof}

Remark: We consider
\begin{align*}
(\mathfrak{A}^{\sharp})^{**}&=(\mathfrak{A}\oplus\mathbb{C})^{**}\\
&=\mathfrak{A}^{**}\oplus\mathbb{C}^{**}\\
&=\mathfrak{A}^{**}\oplus\mathbb{C}=(\mathfrak{A}^{**})^{\sharp}
\end{align*}

\begin{theorem}
Let $\mathfrak{A}$ be a Banach algebra. If $\mathfrak{A}^{**}$ is approximately semi-amenable, then so is
 $\mathfrak{A}$.
\end{theorem}
\begin{proof}
Suppose that $\mathfrak{A}^{**}$ is approximately semi-amenable. Thus by Theorem \ref{mm} there
  are nets $(M_{\alpha})$ and $(N_{\alpha})$ in $((\mathfrak{A}^{**})^{\sharp}\widehat{\otimes
 }(\mathfrak{A}^{**})^{\sharp})^{**}=((\mathfrak{A}^{\sharp})^{**}\widehat{\otimes
 }(\mathfrak{A}^{\sharp})^{**})^{**}$ such that for each
  $F\in(\mathfrak{A}^{**})^{\sharp}=(\mathfrak{A}^{\sharp})^{**}$
 \begin{align*}
 (i)\;\; &F.M_{\alpha}-N_{\alpha}.F\longrightarrow 0\\
 (ii)\;\; &\pi^{**}_{(\mathfrak{A}^{\sharp})^{**}}(M_{\alpha})=\pi^{**}_{(\mathfrak{A}^{\sharp})^{**}}(N_{\alpha})=e
 ,\;\; (\alpha\in\Lambda).
 \end{align*}
 For each $\alpha\in\Lambda$ there are nets $(m_{\alpha}^{(\lambda)})$ and $(n_{\alpha}^{(\lambda)})$ in
  $((\mathfrak{A})^{\sharp})^{**}\widehat{\otimes}((\mathfrak{A})^{\sharp})^{**}$ such that
   $$m_{\alpha}^{(\lambda)}\longrightarrow M_{\alpha},\;\;n_{\alpha}^{(\lambda)}\longrightarrow
    N_{\alpha}$$ in $w^{*}$-topology. Therefore, by Lemma \ref{ss} for continuous linear mapping $\Psi
 :((\mathfrak{A})^{\sharp})^{**}\widehat{\otimes}((\mathfrak{A})^{\sharp})^{**}\longrightarrow
  (\mathfrak{A}^{\sharp}\widehat{\otimes}\mathfrak{A}^{\sharp})^{**}$, the nets
   $(M''_{\alpha,\lambda})=(\Psi(m_{\alpha}^{(\lambda)}))_{\lambda}$ and
    $(N''_{\alpha,\lambda})=(\Psi(n_{\alpha}^{(\lambda)}))_{\lambda}$ are belong to
     $(\mathfrak{A}^{\sharp}\widehat{\otimes}\mathfrak{A}^{\sharp})^{**}$ and for each
      $a\in\mathfrak{A}^{\sharp}$, we have
 \begin{align*}
(i)\;\; a.M''_{\alpha,\lambda}-N''_{\alpha,\lambda}.a&=a
.\Psi(m_{\alpha}^{(\lambda)})-\Psi(n_{\alpha}^{(\lambda)}).a\\
 &=\Psi(am_{\alpha}^{(\lambda)})-\Psi(n_{\alpha}^{(\lambda)}a)\\
 &=\Psi(am_{\alpha}^{(\lambda)}-n_{\alpha}^{(\lambda)}a)\\
 &\rightarrow\Psi(a.M_{\alpha}-N_{\alpha}.a)\\
 &\rightarrow\Psi(0)=0\\
 (ii)\;\; \pi^{**}_{\mathfrak{A}^{\sharp}}(M''_{\alpha,\lambda})&=\pi^{**}_{\mathfrak{A}^{\sharp}
 }(\Psi(m_{\alpha}^{(\lambda)}))\\
 &=\pi_{(\mathfrak{A}^{\sharp})^{**}}(m_{\alpha}^{(\lambda)})\\
 &\rightarrow\pi_{(\mathfrak{A}^{\sharp})^{**}}(M_{\alpha})=e\\
 \pi^{**}_{\mathfrak{A}^{\sharp}}(N''_{\alpha,\lambda})&=\pi^{**}_{\mathfrak{A}^{\sharp}}
 (\Psi(n_{\alpha}^{(\lambda)}))\\
 &=\pi_{(\mathfrak{A}^{\sharp})^{**}}(n_{\alpha}^{(\lambda)})\\
 &\rightarrow\pi_{(\mathfrak{A}^{\sharp})^{**}}(N_{\alpha})=e
 \end{align*}
  Consequently, $\mathfrak{A}$ is approximately semi-amenable.
\end{proof}

    \begin{proposition}
    Let $\mathfrak{A}$ be approximately semi-amenable and $\varphi:\mathfrak{A}\rightarrow\mathcal{B}$ be a
     continuous epimorphism then $\mathcal{B}$ is approximately semi-amenable.
    \end{proposition}
    \begin{proof}
  The proof is straightforward.
    \end{proof}

     \begin{corollary}
    Suppoe that $\mathfrak{A}$ is approximately semi-amenable, and $\mathcal{J}$ is a closed two-sided ideal of
     $\mathfrak{A}$. Then
  the quotient Banach algebra $\mathfrak{A}/\mathcal{J}$ is approximately semi-amenable.
    \end{corollary}

With notice to page 15 from \cite{12}, approximate amenability of $\mathfrak{A}\oplus\mathfrak{A}$ still is open question whenever
 $\mathfrak{A}$ so is. In the following we solve the problem for approximate semi-amenability in general case.
 
 \begin{theorem} $\label{2.9}$
Let $\mathfrak{A}$ and $\mathcal{B}$ be Banach algebras. Then $\mathfrak{A}\oplus \mathcal{B}$ is 
 approximately semi-contractible if and only if $\mathfrak{A}$ and $\mathcal{B}$
 are  approximately semi-contractible.  
\end{theorem} 
\begin{proof}
Let $\mathcal{X}$ and  $\mathcal{Y}$  be respectively, Banach $\mathfrak{A}$-bimodule and Banach
 $\mathcal{B}$-bimodule and $\mathcal{D}_{1}\in \mathcal{Z}^{1}(\mathfrak{A},\mathcal{X})$,
  $\mathcal{D}_{2}\in\mathcal{Z}^{1}(\mathcal{B},\mathcal{Y})$ be continuous. Then $\mathcal{X}\oplus
   \mathcal{Y}$ is a Banach $\mathfrak{A}\oplus \mathcal{B}$-bimodule, by following module actions
\begin{align*}
&(a,b).(x,y):=(ax,by),\\
&(x,y).(a,b):=(xa,yb),
\end{align*}
for $a\in \mathfrak{A}$ , $b\in \mathcal{B}$ , $x\in \mathcal{X}$ and $y\in \mathcal{Y}$. Clearly, the mapping
 $\mathcal{D}^{\sim}:\mathfrak{A}\oplus \mathcal{B} \rightarrow \mathcal{X}\oplus \mathcal{Y}$ specified by
  $\mathcal{D}^{\sim} (a,b):=(\mathcal{D}_{1}(a),\mathcal{D}_{2}(b))$ is a continuous derivation. So by
 assumption there are nets $\lbrace(x_{\alpha} , y_{\alpha})\rbrace_{\alpha}$ and $\lbrace(x'_{\alpha} ,
  y'_{\alpha})\rbrace_{\alpha}$ in
 $X\oplus \mathcal{Y}$ , ( $ x_{\alpha}, x'_{\alpha} \in \mathcal{X}$, and $ y_{\alpha},
  y'_{\alpha} \in\mathcal{Y}$)  such that
 \begin{align*}
\;(\mathcal{D}_{1}(a) , \mathcal{D}_{2}(b)) &=\mathcal{D}^{\sim}(a,b)\\
 &=\lim_{\alpha}((a,b).(x_{\alpha} , y_{\alpha})-(x'_{\alpha}, y'_{\alpha}).(a,b))\\
&=\lim_{\alpha}(ax_{\alpha} - x'_{\alpha}a\,\,  ,\,\,  by_{\alpha} - y'_{\alpha}b). 
\end{align*}
Therefore, $\mathcal{D}_{1}$ and $\mathcal{D}_{2}$ are approximately semi-inner.
 Consequently, $\mathfrak{A}$ and $\mathcal{B}$ are approximately semi-contractible.\\ 
Conversely, let $\mathcal{X}$ be a Banach $\mathfrak{A}\oplus \mathcal{B}$-bimodule and $\mathcal{D}
:\mathfrak{A}\oplus \mathcal{B}\rightarrow \mathcal{X}$ be a continuous derivation. Then $\mathcal{X}$ is a
 Banach $\mathfrak{A}$
 (and $\mathcal{B}$)-bimodule by following module actions,
\begin{align*}
&a.x:=(a,0)x, \;\;\;x.a:=x(a,0), \\
&(\;\;b.x:=(0,b)x, \;\;\; x.b:=x(0,b),)
\end{align*}
for each $a\in \mathfrak{A}$ , $b\in \mathcal{B}$ and $x\in \mathcal{X}$.
Clearly, $\mathcal{D}$ induces two continuous derivations as follows
\begin{align*}
&\mathcal{D}_{1}:\mathfrak{A}\rightarrow \mathcal{X},\;\;\; \mathcal{D}_{1}(a)=\mathcal{D}(a,0), \\
&\mathcal{D}_{2}:\mathcal{B}\rightarrow \mathcal{X}, \;\;\; \mathcal{D}_{2}(b)=\mathcal{D}(0,b).
\end{align*}
So there are nets $(\zeta_{\alpha})_{\alpha}$, $(\xi_{\alpha})_{\alpha}$, $(\mu_{\alpha})_{\alpha}$ and
 $(\eta_{\alpha})_{\alpha}$ in $\mathcal{X}$ such that

\begin{align*}
\;\;\mathcal{D}_{1}(a)&=\lim_{\alpha}[a.\xi_{\alpha}-\zeta_{\alpha}.a]=\lim_{\alpha}[(a,0)
\xi_{\alpha}-\zeta_{\alpha}(a,0)] \\
\mathcal{D}_{2}(b)&=\lim_{\alpha}[b.\mu_{\alpha}-\eta_{\alpha}.b]=\lim_{\alpha}[(0,b)
\mu_{\alpha}-\eta_{\alpha}(0,b)] 
\end{align*}
for each $a\in \mathfrak{A}$ and $b\in \mathcal{B}$ .
\begin{equation}
\mathcal{D}(a,b) =\lim_{\alpha}[(a,0)\xi_{\alpha}-\zeta_{\alpha}(a,0)]+
\lim_{\alpha}[(0,b)\mu_{\alpha}-\eta_{\alpha}(0,b)]. \label{4.1}
\end{equation}

According to approximate semi-contractibility of $\mathfrak{A}$ and $\mathcal{B}$, they have A.I.
 Let $(e_{\alpha})_{\alpha}$ and $(f_{\alpha})_{\alpha}$ be right and left A.I for $\mathfrak{A}$, respectively.
 Also  $(e'_{\alpha})_{\alpha}$ and $(f'_{\alpha})_{\alpha}$ be right and left A.I for $\mathcal{B}$,
  respectively. ( Without lose of generality we used common index set, for all nets ). Then for each $(a,b)\in \mathfrak{A}\oplus
   \mathcal{B}$, we have
\begin{align*}
(a,0) &=\lim_{\alpha}(a,b).(e_{\alpha},0)=\lim_{\alpha}(f_{\alpha},0).(a,b),\\
(0,b) &=\lim_{\alpha}(a,b).(0,e'_{\alpha})=\lim_{\alpha}(0,f'_{\alpha}).(a,b). 
\end{align*}
 
Now by above equation we have 
\begin{align*}
\;\mathcal{D}(a,b)&=\lim_{\alpha} \left[ (a,b). ((e_{\alpha},0)\xi_{\alpha}+(0,e'_{\alpha})
\mu_{\alpha})-(\zeta_{\alpha}(f_{\alpha},0)+\eta_{\alpha}(0,f'_{\alpha})).(a,b)\right].
\end{align*}
Letting,
\begin{align*}
\; (\psi_{\alpha})_{\alpha}&=\lbrace(e_{\alpha},0)\xi_{\alpha}+(0,e'_{\alpha})
\mu_{\alpha}\rbrace_{\alpha}\\
 (\varphi_{\alpha})_{\alpha}&= \lbrace \zeta_{\alpha}
(f_{\alpha},0)+\eta_{\alpha}(0,f'_{\alpha})\rbrace _{\alpha}.
\end{align*}
  Thus the nets $(\psi_{\alpha})$ and $(\mu_{\alpha})$ are belong to $\mathcal{X}$ and  
\begin{align*}
\;\mathcal{D}(a,b)&=\lim_{\alpha}\left[(a,b).\psi_{\alpha}-\varphi_{\alpha}.(a,b)\right].  
\end{align*} 
So $\mathcal{D}$ is approximately semi-inner and consequently, $\mathfrak{A}\oplus
 \mathcal{B}$ is approximately semi-contractible.
\end{proof}

\begin{example}
Let $l^{p}$ and $l^{q}$, $(1\leq p,q<\infty)$ be Banach sequence algebras. By discussion at the begining of this
 section both $l^{p}$ and $l^{q}$ are approximately semi-amenable, then by Theorem \ref{2.9}, $l^{p}\oplus
  l^{q}$, and generally, $\oplus_{i=1}^{n} l^{p}$ are approximately semi-amenable which never approximately
   amenable.
\end{example}
 In the following theorem we use some techniques, from argument of Theorem 3.3 in \cite{2}.
     
\begin{theorem}\label{ff} 
Let $\mathfrak{A}$ and $\mathcal{B}$ be Banach algebras. Then $\mathfrak{A}\widehat{\otimes}\mathcal{B}$ is
 approximately semi-contractible, if and only if  $\mathfrak{A}$ and $\mathcal{B}$ are approximately 
 semi-contractible. 
\end{theorem}

\begin{proof}
Let $\mathcal{X}$ be a Banach $\mathfrak{A}$-bimodule, and $\mathcal{D}: \mathfrak{A}\rightarrow
 \mathcal{X}$ be a continuous derivation. So $\mathcal{X}\widehat{\otimes} \mathcal{B}$ is a
  $\mathfrak{A}\widehat{\otimes} \mathcal{B}$-bimodule by following module actions
\begin{align*}
&(x'\otimes b')(a\otimes b):=x'a\otimes b'b\\
&(a\otimes b)(x'\otimes b'):=ax'\otimes bb'\
\end{align*}
for $a\in \mathfrak{A}$, $x'\in \mathcal{X}$ and $b, b' \in \mathcal{B}$.\\
Define
\begin{align*}
&\mathcal{D}^{\sim}: \mathfrak{A}\widehat{\otimes}\mathcal{B}\rightarrow
 \mathcal{X}\widehat{\otimes}\mathcal{B}\\
&\mathcal{D}^{\sim}(a\otimes b)=\mathcal{D}(a)\otimes b 
\end{align*}
for $a\in \mathfrak{A}$ and $b\in \mathcal{B}$. Then 
\begin{align*}
\mathcal{D}^{\sim}((a_{1}\otimes b_{1})(a_{2}\otimes b_{2}))&=\mathcal{D}^{\sim}(a_{1}a_{2}\otimes
 b_{1}b_{2})\\
&=(\mathcal{D}(a_{1}).a_{2}+a_{1}.\mathcal{D}(a_{2}))\otimes (b_{1}b_{2})\\
&=((\mathcal{D}(a_{1}).a_{2})\otimes b_{1}b_{2})+(a_{1}.\mathcal{D}(a_{2})\otimes b_{1}b_{2})\\
&=((\mathcal{D}(a_{1})\otimes b_{1}).(a_{2}\otimes b_{2}))+((a_{1}\otimes b_{1}).(\mathcal{D
}(a_{2})\otimes b_{2})\\
&=\mathcal{D}^{\sim}(a_{1}\otimes b_{1}).(a_{2}\otimes b_{2})+(a_{1}\otimes b_{1}).\mathcal{D}^{\sim
}(a_{2}\otimes b_{2}).
\end{align*}
Consequently, $\mathcal{D}^{\sim}$ is a derivation. By assumption, $\mathcal{D}^{\sim}$ must be
 approximately semi- inner, so there exist nets $m_{\alpha}=\sum_{n=1}^{\infty} x_{n,\alpha}\otimes b_{n
 ,\alpha}$ and $ n_{\alpha}=\sum_{n=1}^{\infty} x'_{n,\alpha}\otimes b'_{n,\alpha}$ in
  $\mathcal{X}\widehat{\otimes}\mathcal{B}$, $ (x_{n,\alpha}, x'_{n,\alpha}\in \mathcal {X}, b_{n,\alpha}, b'_{n
  ,\alpha}\in \mathcal {B})$, such that, for each $a\otimes b\in\mathfrak{A} \widehat{\otimes} \mathcal{B}$, we
   have
\begin{align*}
\mathcal{D}(a)\otimes b&=\mathcal{D}^{\sim}(a\otimes b) \\
&=\lim_{\alpha}[(a\otimes b).m_{\alpha}-n_{\alpha}.(a\otimes b)]\\
&=\lim_{\alpha}[\sum_{n=1}^{\infty}(a\otimes b).(x_{n,\alpha}\otimes b_{n,\alpha})-\sum_{n=1}^{\infty
}(x'_{n,\alpha}\otimes b'_{n,\alpha}).(a\otimes b)]\\
 &=\lim_{\alpha}\left[ \sum_{n=1}^{\infty}ax_{n,\alpha}\otimes bb_{n,\alpha}-\sum_{n=1}^{\infty}x'_{n
 ,\alpha}a\otimes b'_{n,\alpha}b\right] .
\end{align*}

Now fix $b_{0}\in \mathcal{B}$ non-zero, and take $b^{*}_{0}\in \mathcal{B}^{*}$ with $\langle
 b_{0}^{*},b_{0}\rangle =1.$ Define the operator
\begin{align*}
&\psi:\mathcal{X}\widehat{\otimes}\mathcal{B} \rightarrow \mathcal{X}\\
&\psi(x \otimes b)=\langle b_{0}^{*},b \rangle x
\end{align*}
and by applying it to both sides of equation \eqref{2.5} we have
\begin{align*}
\langle b_{0}^{*},b\rangle\mathcal{D}(a)&=\psi (\mathcal{D}(a)\otimes b)\\
&=\lim_{\alpha}[\sum_{n=1}^{\infty} \psi(ax_{n,\alpha}\otimes bb_{n,\alpha})-\sum_{n=1}^{\infty} \psi(x'_{n
,\alpha}a\otimes b'_{n,\alpha}b)].\\
&=\lim_{\alpha}[\sum_{n=1}^{\infty} a \langle b_{0}^{*} ,bb_{n,\alpha}\rangle x_{n,
\alpha}-\sum_{n=1}^{\infty}\langle b_{0}^{*},b'_{n,\alpha}b\rangle x'_{n,\alpha}a].
\end{align*}
Take $b=b_{0}$, then for each $a\in \mathfrak{A}$,
\begin{align*}
 \mathcal{D}(a)&=\lim_{\alpha} [a(\sum_{n=1}^{\infty} \langle b_{0}^{*} ,b_{0}b_{n,\alpha}\rangle
  x_{n,\alpha})-(\sum_{n=1}^{\infty}\langle b_{0}^{*},b'_{n,\alpha}b_{0}\rangle x'_{n,\alpha})a].
  \end{align*}
By setting $\mu_{\alpha}=\sum_{n=1}^{\infty} \langle b_{0}^{*} ,b_{0}b_{n,\alpha}\rangle x_{n,\alpha}$, and
 $\xi_{\alpha}=\sum_{n=1}^{\infty}\langle b_{0}^{*},b'_{n,\alpha}b_{0}\rangle x'_{n,\alpha}$, so
  $\mu_{\alpha}, \xi_{\alpha}\in \mathcal{X}$ and $\mathcal{D}(a)=\lim_{\alpha} (a
  .\mu_{\alpha}-\xi_{\alpha}.a)$.
Consequently, $\mathfrak{A}$ is approximately semi-contractible. There is a similar argument for $\mathcal{B}$.\\
Conversely, let $\mathfrak{A}$ and $\mathcal{B}$ be approximately semi-amenable. So by Theorem \ref{mm}
 there are nets $(M_{\alpha})$ and $(N_{\alpha})$ in
 $(\mathfrak{A}^{\sharp}\widehat{\otimes}\mathfrak{A}^{\sharp})^{**}$              and nets $(M'_{\alpha})$,
  $(N'_{\alpha})$ in $(\mathcal{B}^{\sharp}\widehat{\otimes}\mathcal{B}^{\sharp})^{**}$ which satisfy
   conditions $(i)$, $(ii)$ and $(iii)$ mentioned in Theorem \ref{mm}. Let $\Psi$ and $\Psi_{1}$ be continuous
    linear maps mentioned in Lemma \ref{sm}, as follows,
\begin{align*}
&\Psi:(\mathfrak{A}^{\sharp}\widehat{\otimes}\mathfrak{A}^{\sharp})^{**}\widehat{\otimes
}(\mathcal{B}^{\sharp}\widehat{\otimes}\mathcal{B}^{\sharp})^{**}\rightarrow
 ((\mathfrak{A}^{\sharp}\widehat{\otimes}\mathfrak{A}^{\sharp})\widehat{\otimes
}(\mathcal{B}^{\sharp}\widehat{\otimes}\mathcal{B}^{\sharp}))^{**}\cong
 ((\mathfrak{A}^{\sharp}\widehat{\otimes}\mathcal{B}^{\sharp})\widehat{\otimes
 }(\mathfrak{A}^{\sharp}\widehat{\otimes}\mathcal{B}^{\sharp}))^{**},\\
&\Psi_{1}:(\mathfrak{A}^{\sharp})^{**}\widehat{\otimes
}(\mathfrak{B}^{\sharp})^{**}\rightarrow(\mathfrak{A}^{\sharp}\widehat{\otimes}
\mathfrak{B}^{\sharp})^{**}.
\end{align*}  
Then $(M''_{\alpha})=\left\lbrace \Psi(M_{\alpha}\otimes M'_{\alpha})\right\rbrace_{\alpha} $ and
 $(N''_{\alpha})=\left\lbrace \Psi(N_{\alpha}\otimes
 N'_{\alpha})\right\rbrace_{\alpha} $ are nets in
  $((\mathfrak{A}^{\sharp}\widehat{\otimes}\mathcal{B}^{\sharp})\widehat{\otimes
 }(\mathfrak{A}^{\sharp}\widehat{\otimes}\mathcal{B}^{\sharp}))^{**}$. For each $(a\otimes b)\in
  \mathfrak{A}^{\sharp}\widehat{\otimes}\mathcal{B}^{\sharp}$ we have, the following items
\begin{align*}
(i)\;\;\;(a\otimes b).M''_{\alpha} - N''_{\alpha}.(a\otimes b)&=(a\otimes b).\Psi(M_{\alpha}\otimes
 M'_{\alpha})-\Psi(N_{\alpha}\otimes N'_{\alpha}).(a\otimes b)\\
&=\Psi(a.M_{\alpha}\otimes b.M'_{\alpha})-\Psi(N_{\alpha}.a\otimes N'_{\alpha}.b)\\
&=\overline{T}(aM_{\alpha},bM'_{\alpha})-\overline{T}(N_{\alpha}a,N'_{\alpha}b)\\
&=\overline{T}(aM_{\alpha}-N_{\alpha}a,bM'_{\alpha}-N'_{\alpha}b)\rightarrow\overline{T}(0,0)=0
\end{align*}
 $(ii)$ It's easily verified that, $\pi^{**}_{\mathfrak{A}^{\sharp}
 \widehat{\otimes}\mathcal{B}^{\sharp}}\circ\Psi=\Psi_{1}\circ(\pi^{**}_{\mathfrak{A}^{\sharp}}\otimes\pi^{**}_{\mathfrak{B}^{\sharp}})
 $. Then we have
\begin{align*}
\pi^{**}_{\mathfrak{A}^{\sharp}\widehat{\otimes}\mathcal{B}^{\sharp}
}(M''_{\alpha})&=\pi^{**}_{\mathfrak{A}^{\sharp}\widehat{\otimes}\mathcal{B}^{\sharp}
}(\Psi(M_{\alpha}\otimes M'_{\alpha}))\\
&=(\Psi_{1}\circ(\pi^{**}_{\mathfrak{A}^{\sharp}}\otimes\pi^{**}_{\mathfrak{B}^{\sharp}})
(M_{\alpha}\otimes M'_{\alpha})\\
&=\Psi_{1}((\pi^{**}_{\mathfrak{A}^{\sharp}}(M_{\alpha})\otimes (\pi^{**}_{\mathfrak{B}^{\sharp}
}(M'_{\alpha}))\\
&=\Psi_{1}(e_{_{(\mathfrak{A}^{\sharp})^{**}}}\otimes e_{_{(\mathfrak{B}^{\sharp})^{**}}})\\
&=e_{_{(\mathfrak{A}^{\sharp}\widehat{\otimes}\mathcal{B}^{\sharp})^{**}}}
\end{align*}

$(iii)$ Similarly, $\pi^{**}_{\mathfrak{A}^{\sharp}\widehat{\otimes}\mathcal{B}^{\sharp}
}(N''_{\alpha})=e_{_{(\mathfrak{A}^{\sharp}\widehat{\otimes}\mathcal{B}^{\sharp})^{**}}}$.
\end{proof}

Since for a non-zero character $\varphi$ on $l^{p}$; the map $x \otimes y \rightarrow \varphi(x)y$ determines a
 continuous epimorphism of $l^{p}\widehat{\otimes} l^{p}$ onto $l^{p}$. So if $l^{p}\widehat{\otimes} l^{p}$
 is amenable, $l^{p}$ would also be amenable, which is contrary to non-amenability of $l^{p}$. On the one hand
  by Theorem 2.1, of \cite{4} approximate amenability is equivalent by approximate contractibility. Thus
   $l^{p}\widehat{\otimes}l^{p}$ is not approximate contractible and equivalently approximate amenable.
\begin{example}
 By Theorem $\ref{ff}$, $l^{p}\widehat{\otimes} l^{q}$ for $1\leq p,q <\infty$ is approximately semi-amenable. Which is never
  approximately amenable.
\end{example}

\begin{definition}
A dual Banach algebra $\mathfrak{A}$ is called approximately Connes semi-amenable if, for every normal, dual
 Banach $\mathfrak{A}$-bimodule $\mathcal{X}$, every $w^{*}$-continuous derivation $\mathcal{D}:
 \mathfrak{A}\rightarrow\mathcal{X}$ is approximately semi-inner. 
\end{definition}

\begin{theorem}
Let $\mathfrak{A}$ and $\mathcal{B}$ be dual Banach algebras.
Then $\mathfrak{A}\oplus \mathcal{B}$ is approximately Connes semi-amenable   if and only if 
$\mathfrak{A}$ and $\mathcal{B}$ are approximately Connes semi-amenable.
\end{theorem}
\begin{proof}
 Let $\mathcal{X}$ and $\mathcal{Y}$ be normal, dual and Banach $\mathfrak{A}$ and 
 $\mathcal{B}$-bimodule, respectively. Let $\mathcal{D}_{1}\in \mathcal{Z}^{1}_{w^{*}}(\mathfrak{A}
 ,\mathcal{X})$ and $\mathcal{D}_{2}\in\mathcal{Z}^{1}_{w^{*}}(\mathcal{B},\mathcal{Y})$. So
  $\mathcal{X}\oplus \mathcal{Y}$ is normal, dual and Banach $\mathfrak{A}\oplus \mathcal{B}$-bimodule by
   module operations defined with $$(a,b).(x,y):=(ax,by),\; (x,y).(a,b):=(xa,yb).$$ The mapping
\begin{align*}
&\mathcal{D}^{\sim}:\mathfrak{A}\oplus \mathcal{B}\rightarrow \mathcal{X}\oplus \mathcal{Y}\\
&\mathcal{D}^{\sim}(a,b)=(\mathcal{D}_{1}(a),\mathcal{D}_{2}(b))\
\end{align*}
  is a derivation, since 
 \begin{align*} 
\mathcal{D}^{\sim}[(a,b)(c,d)]&=\mathcal{D}^{\sim}(ac,bd)\\
&=(\mathcal{D}_{1}(ac),\mathcal{D}_{2}(bd))\\
 &=(\mathcal{D}_{1}(a).c+a.\mathcal{D}_{1}(c),\mathcal{D}_{2}(b).d+b.\mathcal{D}_{2}(d))\\ 
  [ \mathcal{D}^{\sim}(a,b)].(c,d)&=(\mathcal{D}_{1}(a),\mathcal{D}_{2}(b)).(c,d)\\
   &=(\mathcal{D}_{1}(a).c,\mathcal{D}_{2}(b).d)\\ 
  (a,b).[\mathcal{D}^{\sim}(c,d)]&=(a,b).(\mathcal{D}_{1}(c),\mathcal{D}_{2}(d))\\
  &=(a.\mathcal{D}_{1}(c),b.\mathcal{D}_{2}(d)).
   \end{align*}
 Thus, $$\mathcal{D}^{\sim}[(a,b)(c,d)]=[\mathcal{D}^{\sim}(a,b)].(c,d)+(a,b).[\mathcal{D}^{\sim}(c,d)]$$
 and by $w^{*}$-continuity of $\mathcal{D}_{1}$ and $\mathcal{D}_{2}$ we conclude $\mathcal{D}^{\sim} 
 \in\mathcal{Z}^{1}_{w^{*}}(\mathfrak{A}\oplus \mathcal{B},\mathcal{X}\oplus\mathcal{Y})$.
   So due to the approximate Connes semi-amenability of $\mathfrak{A}\oplus \mathcal{B}$,
    $\mathcal{D}^{\sim}$ is approximately semi-inner. Thus
    there are nets  $(\psi_{\alpha})=(x_{\alpha},y_{\alpha})_{\alpha}$ and
     $(\mu_{\alpha})=(x'_{\alpha},y'_{\alpha})_{\alpha}$, in $\mathcal{X}\oplus\mathcal{Y}$ such that for each
$(a,b)\in \mathfrak{A}\oplus \mathcal{B}$ we have 
 \begin{align*}
(\mathcal{D}_{1}(a),\mathcal{D}_{2}(b))&=\mathcal{D}^{\sim}(a,b)\\
&=\lim_{\alpha}\left\lbrace (a,b).\psi_{\alpha}-\mu_{\alpha}.(a,b)\right\rbrace \\
 &=\lim_{\alpha} \left\lbrace (a,b).(x_{\alpha},y_{\alpha})-(x'_{\alpha},y'_{\alpha}).(a,b)\right\rbrace \\
&=\lim_{\alpha}  (ax_{\alpha}-x'_{\alpha}a,by_{\alpha}-y'_{\alpha}b).
 \end{align*}
 Consequently, $\mathcal{D}_{1}$ and $\mathcal{D}_{2}$ are approximately semi-inner. Thus $\mathfrak{A}$
  and $\mathcal{B}$ are approximately Connes semi-amenable .\\
 Conversely, let $\mathcal{X}$ be a normal, dual Banach $\mathfrak{A}\oplus\mathcal{B}$-bimodule and
  $\mathcal{D}\in \mathcal{Z}^{1}_{w^{*}}(\mathfrak{A}\oplus\mathcal{B},\mathcal{X})$. Then $\mathcal{X}$
   is a normal, dual Banach $\mathfrak{A}$ (and $\mathcal{B}$)-bimodule by following multiplications,
  \begin{align*}
  &a.x:=(a,0)x,\;x.a:=x(a,0),\\
  (\text{and} \;&b.x:=(0,b)x,\;x.b:=x(0,b)).
  \end{align*}
 We define $\mathcal{D}_{1}:\mathfrak{A}\rightarrow \mathcal{X}$ with $\mathcal{D}_{1}(a)=\mathcal{D
 }(a,0)$. Obviously, $\mathcal{D}_{1}$ is a $w^{*}$-continuous derivation. Thus according to approximately
  Connes semi-amenability of $\mathfrak{A}$, $\mathcal{D}_{1}=\lim_{\alpha}\delta_{y_{\alpha}}$, for some net
 $(y_{\alpha})$ in $\mathcal{X}$. Therefore
    $\mathcal{D}^{\sim}=(\mathcal{D}-\lim_{\alpha}\delta_{y_{\alpha}})\in
    \mathcal{Z}^{1
  }_{w^{*}}(\mathfrak{A}\oplus\mathcal{B},\mathcal{X})$ and
   $\mathcal{D}^{\sim}\vert_{_{\mathfrak{A}}}=\lbrace0 \rbrace$. Hence $\mathcal{D}^{\sim}
  :\mathcal{B}\rightarrow \mathcal{X}$ is a $w^{*}$-continuous derivation.
   So $\mathcal{D}^{\sim}=\lim_{\alpha}\delta_{z_{\alpha}}$, for some net $(z_{\alpha})$ in $\mathcal{X}$.
    Consequently,
     $\mathcal{D}=\lim_{\alpha}\delta_{y_{\alpha}+z_{\alpha}}$
     and it follows that $\mathfrak{A}\oplus\mathcal{B}$ is approximate Connes semi-amenable.
 \end{proof}

\bibliographystyle{amsplain}

\end{document}